\documentclass[preprint,12pt]{elsarticle}
\usepackage{tikz,tikz-cd,amscd,amsmath, amsthm,amssymb,braket,mathrsfs,graphicx,verbatim}
\usetikzlibrary{matrix,arrows,decorations.pathmorphing}
\usepackage[all]{xy}
\newtheorem{theorem}{Theorem}[]
\theoremstyle{definition}\newtheorem*{definition}{Definition}
\theoremstyle{definition} \newtheorem{proposition}[theorem]{Proposition}
\theoremstyle{definition}\newtheorem{lemma}[theorem]{Lemma}
\theoremstyle{definition}\newtheorem{corollary}[theorem]{Corollary}
\theoremstyle{definition}
\theoremstyle{remark}\newtheorem*{remark}{Remark}

\newcommand{\sob}[2]{{\Vert {#1}\Vert}_{#2}}
\newcommand{\co}[2]{{| {#1}|}_{#2}}
\DeclareMathOperator{\supp}{supp}
\DeclareMathOperator{\intd}{d\!}

\DeclareMathOperator{\cat}{cat}

\DeclareMathOperator{\const}{const}

\DeclareMathOperator{\length}{length}
\DeclareMathOperator{\gr}{graph}

\DeclareMathOperator{\pd}{PD}
\DeclareMathOperator{\id}{Id}
\def\D{\mathcal{D}}

\def\M{\mathcal{M}}
\def\H{\mathcal{H}}
\def\E{\mathcal{E}}
\def\hk{H^{k}(I,M)}
\def\h2{H^{2}(I,M)}

\journal{J. Mathematical Analysis and Applications}

\begin{document}

\begin{frontmatter}



\title{Existence of variationally defined curves in complete Riemannian manifolds}


\author{Philip Schrader\corref{cor1}}
\cortext[cor1]{Corresponding Author.}
\ead{10428752@student.uwa.edu.au}
\address{School of Mathematics and Statistics, The University of Western Australia, Crawley WA 6009, Australia.}

\begin{abstract}
We present a method for proving the existence of solutions to a class of one dimensional variational problems. The method is demonstrated by two examples of optimal interpolation problems which are motivated by engineering applications. In each case we prove that the variational problem satisfies the Palais-Smale condition and the existence of a minimal solution and lower bounds for the number of stationary curves then follow. 

\end{abstract}


\end{frontmatter}

\section{Introduction}
\paragraph{Motivation}
Recently there has been increased interest in variationally defined curves with higher order derivatives in the Lagrangian for the purpose of intrinsic interpolation in manifolds. This began in \cite{Gabriel:1985zp,noakes:cscs} with the study of stationary paths of the average covariant acceleration $\tfrac{1}{2}\int\Vert \nabla_{t}\dot x\Vert^{2} dt$, where $x:I\to M$ is a map from the unit interval to a Riemannian manifold $M$. In Euclidean space these curves are cubic polynomials; they are therefore known as Riemannian cubics, and are a natural candidate for interpolation when we require differentiability at knot points. For example, in trajectory planning for oriented rigid bodies rapid changes of direction are inefficient so it is desirable for paths to be at least $C^{1}$. Applications of this kind require interpolation in the Lie group $SO(3)$ of 3D orientations (or SE(3), the group of rigid body motions), and have been the primary motivation for studying Riemannian cubics. More recently Riemannian cubics have also been studied for interpolation on spaces of shapes (images, landmarks, curves, surfaces or tensors) in computational anatomy, motivated by applications in medical imaging \cite{Gay-Balmaz:2012wq, Gay-Balmaz:2012ad, Trouve:2012ao}.

Several alternative interpolation schemes have been proposed in the interim (see for example \cite{Machado:2010zm} or  \cite{noakes:tension} and the references therein), most of which are variational in nature. Of particular interest for this paper are the so called \emph{Riemannian cubics in tension} and \emph{conditional extremals}. Cubics in tension are the stationary paths of $\int \Vert\nabla_{t}\dot x\Vert^{2}+\tau^{2}\Vert^{} \dot x\Vert^{2}dt$, where $\tau$ is a constant known as the tension parameter. A special case of these curves is analysed qualitatively in \cite{noakes:tension}. They have also been studied from an optimal control point of view in \cite{Silva-Leite:kb,Silva-Leite:2000qa} where they are called elastic curves, and also \cite{Hussein:2004ye, Hussein:2007rz} where they are called $\tau$-elastic curves and applied to spacecraft interferometric imaging. In this particular application acceleration requires fuel expenditure and image quality is inversely related to speed, so it makes sense to penalise both.
Conditional extremals are the stationary paths of $\int \Vert\dot x-A(x)\Vert^{2} dt$ where $A$ is a fixed vector field \cite{Noakes:2012sf,Schrader:2012fk}. These curves are motivated by a hypothetical problem in which $\dot x =A$ is an estimate for the equations of motion of a system, but is not compatible with the observed data. Conditional extremals are then an estimate of the actual trajectories.

So far the only existence results for higher order interpolants are those of Giamb\`o et. al. for Riemannian cubics and higher order geodesics \cite{Giambo:2002wd,Giambo:2004nx}. The purpose of this paper is to present a general method which includes the previous existence results and apply it to prove existence of cubics in tension and conditional extremals.

\paragraph{Content of the paper}
From the global analytical point of view a variational problem of the kind discussed above consists of an action functional $f$ and a Hilbert manifold $X$ of curves satisfying some boundary conditions, which is a submanifold of the natural domain of the action. The standard criterion for existence of critical points is the Palais-Smale condition: $f:X\to \mathbb R$ is said to saisfy the Palais-Smale condition if any sequence $(x_{i})\subset X$ on which $f$ is bounded, and for which $|df(x_{i})|\to 0$, has a convergent subsequence.  If this condition is satisfied then we are guaranteed not only existence of critical points of $f$ but also existence of a minimum (see eg. \cite{Palais:1988fv}). Furthermore we obtain lower bounds for the total number of critical points via Ljusternik-Schnirelman theory, and possibly also Morse theory. 

El{\'{\i}}asson has in fact verified the Palais-Smale condition for a large class of variational problems \cite{Eluiasson:1970cr, Eliasson:1974sf}. The results of Giamb\`o et. al. and the variational problems we consider in this paper do not quite fit into the class of problems treated therein. Nevertheless we approach the problem using Eliasson's observation that verification of the Palais-Smale condition can be conveniently separated into three parts \cite{Eliasson:1971qf}. One verifies that the action is \emph{weakly proper} on $X$ with respect to some larger manifold $X_{0}$ containing $X$ as a \emph{weak submanifold}, and this allows the problem to be treated in local coordinates for $X$. It remains to confirm that the action functional is \emph{locally bounding} and \emph{locally coercive} on $X$ with respect to $X_{0}$ and the Palais-Smale condition follows (see section \ref{eliassonsmethod} for precise definitions).

It is proved in \cite{Eliasson:1974sf} that the action functionals with so-called \emph{strongly elliptic} Lagrangians are locally bounding and locally coercive, with respect to the Banach manifold of continuous maps $C^{0}(I,M)$, on the domain consisting of curves satisfying boundary conditions of sufficiently high order. This assumption on the domain is not compatible with, for example, the problem of $C^{1}$ optimal interpolation of $n$ points by cubics in tension. We therefore prove in section \ref{mainresults} that action functionals of this kind are locally coercive and locally bounding on their \emph{natural domain} with respect to $C^{k}(I,M)$. Moreover we give some conditions for submanifolds of the natural domain on which the restriction of the action functional remains locally coercive and locally bounding. These results are used to verify the Palais-Smale condition for the two variational problems discussed previously: higher order conditional extremals and cubics in tension. We confirm in each case that the problem is weakly proper with respect to $C^{k}(I,M)$. These two examples demonstrate the balance between action functional and boundary conditions (i.e. domain) which is required for the problem to be weakly proper.

\section{Preliminaries}\label{preliminaries}
In this section we give an overview of the relevant material from \cite{Eluiasson:1967rr,Eluiasson:1970cr, Eliasson:1971qf, Eliasson:1974sf}, including the method for establishing the Palais-Smale condition and the construction of charts for manifolds of maps. We also define the class of Lagrangians for which we will later prove that the action functional is locally bounding and locally coercive: the strongly elliptic polynomial differential operators.

\subsection{Banach manifolds}
The definition of a Banach manifold is essentially the same as the usual definition of a manifold, except that we allow charts to map homeomorphically onto subsets of a Banach space instead of requiring that they map into $\mathbb R^{n}$. For this and other basic definitions we refer to \cite{Lang:1999sf}. The fact that a closed subspace of a Banach space need not split (that is, may not have a closed complementary subspace) necessitates some changes to the typical definitions of submersions and immersions:
\begin{itemize}
\item[-] a function $f:X\to Y$ between Banach manifolds is a \emph{submersion} at $x\in X$ if $T_{x}f$ is surjective and its kernel splits
\item[-]$f$ is called an \emph{immersion} at $x\in X$ if $T_{x}f$ is injective and its image splits.
\end{itemize}
This leaves the possibility of defining a weak immersion by dropping the assumption of a splitting image (cf. \cite{Odzijewicz:2003qf}). In a similar vein we have the following definition of a \emph{weak submanifold} from \cite{Eliasson:1971qf}.

\begin{definition}
Let $X,X_{0}$ be Banach manifolds modelled on $B,B_{0}$ respectively, and suppose $X\subset X_{0}, B\subset B_{0}$ with the latter a continuous linear inclusion. Then $X$ is a \emph{weak submanifold} of $X_{0}$ if for any $x_{0}$ in the closure of $X$ there is a chart $\phi_{0}:U_{0}\to \phi_{0}(U)\subset B_{0}$ for $X_{0}$ containing $x_{0}$, such that the restriction of $\phi_{0}$ to $U:=U_{0}\cap X$ is a chart $\phi: U\to \phi(U)\subset B$ for $X$. Any chart for $X$ which arises in this way will be called a \emph{weak chart at $x_{0}$}. 
\end{definition}
Note that this definition allows the topology of the weak submanifold to be finer than the relative topology.

\subsection{Method for establishing the Palais-Smale condition}\label{eliassonsmethod}

A Finsler structure on a Banach manifold $X$ is a continuous function $v\mapsto \Vert v\Vert$ on the tangent bundle $\tau:TX\to X$ such that the restriction to each fibre $T_{x}X$ is a norm, and such that in any local trivialisation $\Phi: \tau^{-1}U\to \phi(U)\times B$, and for any constant $k>1$,  we have 
\begin{equation}\label{uniform}
 \tfrac{1}{k}\Vert\Phi^{-1}(\xi,\eta)\Vert \leq \Vert \Phi^{-1}(\xi_{0},\eta)\Vert\leq k\Vert \Phi^{-1}(\xi,\eta)\Vert 
 \end{equation}
with $\eta\in B$ and $\xi$ sufficiently close to $\xi_{0}=\phi(x)$. That is, the fibre norms are locally equivalent. 

Suppose $X$ is a weak submanifold of $X_{0}$ and let ${\Vert \,\Vert}_{B}$ denote the norm for $B$ and ${|\,|}_{0}$ the norm for $B_{0}$. We call a Finsler structure on $X$ \emph{locally bounded} with respect to $X_{0}$ if for any $x_{0}$ in the closure $\bar X$ and any constant $L$, there is a local trivialisation $\Phi$ over a weak chart $\phi$  at $x_{0}$ and a constant $c$ such that
\[\Vert \Phi^{-1}(\xi,\eta)\Vert\leq c {\Vert\eta\Vert}_{B} \]
for all $\xi\in \phi(U)$ with ${\Vert \xi\Vert}_{B}<L$, and all $\eta \in B$. 

 A function $f:X\to \mathbb{R}$ is called
\begin{itemize}
\item[(i)]  \emph{weakly proper} with respect to $X_{0}$ if any subset $A\subset X$ on which $f$ is bounded is relatively compact in $X_{0}$.
\item[(ii)]  \emph{locally bounding}
 with respect to $X_{0}$ if for any constants $K,L$ and $x_{0}\in \bar X$ there is a weak chart $(U,\phi)$ at $x_{0}$ and a constant $\alpha$ such that for all $\xi\in \phi(U)$ with ${|\xi|}_{0}<K$ and $f_{\phi}(\xi):=f(\phi^{-1} (\xi))<L$, we have ${\Vert \xi \Vert}_{B} < \alpha$.
 \item[(iii)] \emph{locally coercive} with respect to $X_{0}$ if it is $C^{1}$ and for any $x_{0}\in \bar X$ and any constant $K$, there is a weak chart $(U,\phi)$ at $x_{0}$ and there exist constants $\lambda>0,c$ such that 
\begin{equation}\label{lc1}
 (Df_{\phi}(\xi)-Df_{\phi}(\eta))(\xi-\eta)\geq\lambda{\Vert \xi-\eta\Vert}_{B}^{2}-c{| \xi-\eta |}_{0}^{2}
\end{equation}
for all $\xi,\eta\in \phi(U)$ with ${\Vert \xi\Vert}_{B}<K,{\Vert \eta\Vert}_{B}<K$. If $f$ is of class $C^{2}$ we have an equivalent condition:
\begin{equation}\label{lc2}
D^{2}f_{\phi}(\xi)(\eta,\eta)\geq \lambda{\Vert\eta\Vert}_{B}^{2}-c{|\eta|}^{2}_{0} 
\end{equation}
for all $\xi\in \phi(U)$ with ${\Vert\xi\Vert}_{B} <K$ and all $\eta\in B$.
 \end{itemize}

 The assumption of an upper bound for ${|\xi|}_{0}$  is not included in the original definition of locally bounding \cite{Eliasson:1971qf} because it does not need to be assumed if $f$ is weakly proper. However we will find it useful to be able to prove that $f$ is locally bounding independently.

\begin{theorem} (El\'{i}asson \cite{Eliasson:1974sf})\label{psc1}
Let $X$ be a weak submanifold of $X_{0}$ as above, with locally bounded Finsler structure.
If $f:X\to\mathbb{R}$ is of class $C^{1}$; and weakly proper, locally bounding and locally coercive each with respect to $X_{0}$, then $f$ satisfies the Palais-Smale condition.
\end{theorem}

\subsection{Manifolds of Maps}

Throughout this paper we let $S$ be either the unit interval $I$ or the circle $I/\{0,1\}$ and $M$ a complete Riemannian manifold of class $C^{\infty}$ and with finite dimension. The constructions of manifolds of maps and polynomial differential operators which follow are all given for compact $S$ with finite dimension in \cite{Eluiasson:1967rr, Eliasson:1971qf}. However for now we are only interested in applications where $S$ is one-dimensional, and making this assumption simplifies the exposition considerably.

\paragraph{Manifolds of continuous maps}
We begin with a brief outline of the construction of a differentiable structure for the space $C^{0}(S,M)$ of continuous curves $S\to M$ with the compact-open topology.  We will omit many of the details; for a thorough exposition we refer to \cite{Eluiasson:1967rr}.

  Let $h:S\to M$ be smooth and consider the pull-back bundle $h^{*}TM$ with metric and connection induced by those on $M$.
\[
\begin{CD}
h^{*}TM @>\tau^{*}h>> TM\\
@VVh^{*}\tau V @VV\tau V\\
S @>h>> M
\end{CD}
\]
   The model space for charts on $C^{0}(S,M)$ is the Banach space $C^{0}(h^{*}TM)$ of continuous sections of $h^{*}TM$, with norm 
\[|\xi|_{0}=\sup_{s\in S}\Vert \xi(s)\Vert,\quad \xi\in C^{0}(h^{*}TM) \]
 where $\Vert \,\Vert$ is the norm on $h^{*}TM$ pulled back from the Riemannian metric on $M$. 

Let $\exp:TM\to M$ be the exponential map obtained from the Levi-Civita connection on $M$. Since $\tau^{*}h$ is the identity on fibres we will abbreviate  $\exp\circ\tau^{*}h$ to $\exp_{h}$. It is a standard result from Riemannian geometry that there exists an open neighbourhood $ \D$ of the zero section in $TM$ such that $(\tau,\exp):\D\to M\times M$ is a diffeomorphism onto an open neighbourhood of the diagonal. It follows that $(h^{*}\tau,\exp_{h}):h^{*}\D\to S\times M$ is a diffeomorphism onto an open neighbourhood of the graph of $h$ (the image of $(\id,h):S\to S\times M$), and then the map 
\begin{align*} 
C^{0}(\exp_{h}):C^{0}(h^{*}\D)&\to C^{0}(S,M)\\
\xi &\mapsto \exp_{h}\circ \xi
\end{align*}
where $C^{0}(h^{*}\D):=\{\xi\in C^{0}(h^{*}TM):\xi(s)\subset h^{*}\D\}$, is a homeomorphism onto its image, with inverse $C^{0}(\exp_{h})^{-1}(x)= (h^{*}\tau, \exp_{h})^{-1}\circ (id,x)$. Thus $\phi_{h}:=C^{0}(\exp_{h})^{-1}$ is a chart for $C^{0}(S,M)$. The collection $\phi_{h}, h\in C^{\infty}(S,M)$ generates a differentiable structure for $C^{0}(S,M)$.

\paragraph{Some weak submanifolds of $C^{0}(S,M)$}
Define the Banach spaces $C^{k}(h^{*}TM)$, $L^{p}_{k}(h^{*}TM), 1\leq p <\infty$ as the completions of $C^{\infty}(h^{*}TM)$ with respect to the norms
\begin{equation}\label{norm}
 {|\xi|}_{k}:=\sum_{i=0}^{k}\sup_{t\in S}\Vert \nabla_{t}^{i}\xi(t)\Vert \quad {\Vert\xi\Vert}_{p,k}:=\big(\sum_{j=0}^{k}\int_{S}\Vert\nabla_{t}^{j}\xi\Vert^{p}\big)^{1/p}
 \end{equation}
respectively. In the case $p=2$ we have a Hilbert space and we write $H^{k}$ instead of $L_{k}^{2}$, and $\sob{\,}{k}$ instead of ${\Vert\,\Vert}_{2,k}$. The inclusions $H^{k}(h^{*}TM)\subset C^{j}(h^{*}TM)\subset L^{p}_{0}(h^{*}TM), k>j\geq0$, known as Sobolev imbeddings, are all continuous linear and the first is compact.

We write $\M$ to denote either $C^{k}$ or $H^{k}$ (or more generally any \emph{manifold model}, see \cite{Eluiasson:1967rr}), and let $\M (S,M)$ be the set of all $g\in C^{0}(S,M)$ such that there exists $h\in C^{\infty}(S,M)$ with $\gr g \subset (h^{*}\tau,\exp_{h})(h^{*}\D)$ and such that $(h^{*}\tau,\exp_{h})^{-1}\circ (id,g)\in \M(h^{*}\D)$. Then $\M(\exp_{h})={C^{0}(\exp_{h})|}_{\M(h^{*}\D)}$ is a homeomorphism onto its image $U_{h}\subset \M(S,M)$, and $\phi_{h}|U_{h}$ is a chart for $\M (S,M)$  called a \emph{natural chart at $h$}. The collection $(\phi_{h},U_{h}), h\in C^{\infty}(S,M)$ is an atlas for $\M(S,M)$. We observe that a natural chart for $\M (S,M)$ is the restriction of a chart for $C^{0}(S,M)$ to $\M(S,M)$, i.e. $\M (S,M)$ is a weak submanifold of $C^{0}(S,M)$. Similarly $\M'(S,M)$ will be a weak submanifold of $\M(S,M)$ whenever we have a continuous inclusion $\M'(h^{*}TM)\subset \M(h^{*}TM)$. In particular, from the Sobolev imbeddings we have that $H^{k}(S,M)$ is a weak submanifold of $C^{k-1}(S,M)$.

\paragraph{Bundles of sections over manifolds of maps}
Let $\mathcal{F},\mathcal M$ represent either $H^{j},H^{k}$ or $C^{j},C^{k}$ with $j\leq k$.  The union of the Banach spaces $\mathcal{F}(x^{*}TM)$ over all $x\in\mathcal M(S,M)$ can be given the structure of a smooth vector bundle which we denote $\mathcal F(\mathcal M(S,M)^{*}TM)\to \mathcal M(S,M)$.
 A local trivialisation $\Phi_{h}$ over the natural chart $\phi_{h}$ is given by 
\begin{align*} 
\Phi_{h}^{-1}:\M(h^{*}\D)\times \mathcal F(h^{*}TM)&\to \mathcal F(\mathcal M(S,M)^{*}TM) \\
(\xi,\eta)&\mapsto (D_{2}\exp)_{h}(\xi)\eta
\end{align*}
where $D_{2}\exp$ denotes the \emph{fibre derivative} $D_{2}\exp(v):=D(\exp|_{T_{p}M})(v)$ for $v\in T_{p}M $.
If $\M=\mathcal F$ this bundle is equivalent to the tangent bundle $T\mathcal M(S,M)$ and also the bundle $\M(S,TM)$ whose projection is composition with the tangent bundle map $\tau:TM\to M$.

\begin{remark}
The notation $H^{j}(H^{k}(S,M)^{*}TM)$ is descriptive but cumbersome, so we will abbreviate it to $\H^{k,j}$. 
\end{remark}

\subsection{Polynomial Differential Operators}
In \cite{Eluiasson:1970cr} the polynomial differential operators (PDOs) are defined on smooth sections of a fibre bundle $W\to S$. We do not need the same generality, so we follow \cite{Eliasson:1974sf} and assume $W=S\times M$.

Let $E\to S$ be a vector bundle with a Riemannian metric connection and $V\subset E$ an open subset which projects onto $S$. Let $F\to S$ be another vector bundle and suppose we have a smooth fibre preserving map $A_{\alpha}:V\to L^{r}(E,F)$ for each multi-index $\alpha=(\alpha_{1},\ldots,\alpha_{r})$ with  $r\geq 0, 1\leq \alpha_{i}\leq k, \sum \alpha_{i}\leq w$. Then the map $P:C^{\infty}(V)\to C^{\infty}(F)$ defined by 
\[P(\xi)=\sum_{\alpha}A_{\alpha}(\xi)(\nabla_{t}^{\alpha_{1}}\xi,\ldots,\nabla_{t}^{\alpha_{r}}\xi) \]
is called a \emph{local polynomial differential operator}(PDO) from $V$ to $F$ of order $k$ and weight $w$, denoted $P\in \pd_{k}^{w}(V,F)$. Eliasson has proved that these operators have unique continuous extensions to smooth maps between Sobolev spaces of sections. In particular we have the following (cf. Corollary 7.1,7.2 in \cite{Eluiasson:1970cr}) :
\begin{itemize}
\item[(i)] If $P\in \pd_{j}^{j}(V,F)$ and $k\geq j$ then $P$ can be extended to a smooth map $P:H^{k}(V)\to H^{k-j}(F)$.
\item[(ii)] $P\in PD^{2k}_{k}$ can be extended to a smooth map $P:H^{k}(V)\to L_{0}^{1}(F)$
\end{itemize}
These results allow us to make global extensions of maps defined on $C^{\infty}(S,M)$ which can be represented locally by PDOs.
For example, the map $\partial:x\mapsto \dot x$ which takes a smooth curve to its tangent lift has an extension to a smooth map $\partial:H^{k}(S,M)\to \H^{k,k-1}$.  
 In a local trivialisation centred at $h$ the principal part of $\partial$ is the extension of $\partial_{h}\in PD_{1}^{1}(h^{*}\D,h^{*}TM)$, $\partial_{h}(\xi):=\nabla_{t}\xi+\theta(\xi)$ to a smooth map $\partial_{h}:H^{k}(h^{*}\D)\to H^{k-1}(h^{*}TM)$. The precise expression for $\theta$ may be found in \cite{Eluiasson:1967rr} Theorem 6.2, where it is proved that this formula does indeed represent the tangent lift. Similarly, $j$-times covariant differentiation of vector fields along curves extends to a smooth map 
 $\nabla_{t}^{j}:H^{k}(S,M)\to L(\H^{k,r}, \H^{k,r-j})
$
with local principal part
\begin{equation}\label{localrep}
 (\nabla_{t}^{j})_{h}(\xi)\eta=\nabla^{j}_{t}\eta+\sum_{i=0}^{j-1}P_{ij}(\xi)\nabla_{t}^{i}\eta 
 \end{equation}
where $P_{ij}\in \pd_{j-i}^{j-i}(h^{*}\D,L(h^{*}TM,h^{*}TM))$.

If we take the contraction of $\partial$ with $\nabla_{t}^{j-1}$ we obtain a smooth section $\nabla_{t}^{j-1}\partial: H^{k}(S,M)\to \mathcal H^{k,k-j},\, x\mapsto \nabla_{t}^{j-1}\dot x,$ with local principal part
\begin{equation} \label{lf}
(\nabla_{t}^{j-1}\partial)_{h}(\xi)=\nabla_{t}^{j}\xi+Q_{j}(\xi)
\end{equation}
where $Q_{j}\in PD^{j}_{j-1}(h^{*}\D,h^{*}TM)$. This is the basic operator for \emph{(global) polynomial differential operators}, which we define as follows.

Let $\E\to S\times M$ be a vector bundle and suppose we have a smooth section $B_{\alpha}:S\times M\to L^{r}(S\times TM,\E)$ for each multi-index $\alpha=(\alpha_{1},\ldots\alpha_{r})$ with $r\geq 0, 1\leq \alpha_{i}\leq k$ and $\sum_{i}\alpha_{i}\leq w$.
Then given $x\in C^{\infty}(S,M)$ we define a smooth section $P(x):S\to (\id,x)^{*}\E$ by 
\[ P(x)(t):=\sum_{\alpha}B_{\alpha}(t,x(t))(\nabla_{t}^{\alpha_{1}-1}\dot x,\ldots, \nabla_{t}^{\alpha_{r}-1}\dot x) \]
Thus $P$ is a map $C^{\infty}(S,M)\to C^{\infty}(C^{\infty}(S,M)^{*}\E)$ which we call a \emph{polynomial differential operator} on $\E$ of order $k$ and weight $w$, denoted $P\in \pd_{k}^{w}(\E)$. These operators can also be extended to smooth sections.  In particular $P\in PD^{2k}_{k}(\E)$ extends to a smooth map $H^{k}(S,M)\to L^{1}_{0}(H^{k}(S,M)^{*} \E)$. In a local trivialisation $\Phi_{h}$ we have, using \eqref{lf},
\[ P_{h}(\xi)=\sum_{\alpha}A_{\alpha}(\xi)(\nabla_{t}^{\alpha_{1}}\xi,\ldots ,\nabla_{t}^{\alpha_{r}}\xi)\]
for some $A_{\alpha}:h^{*}\D\to L^{r}(h^{*}TM,h^{*}\E)$
so $P_{h}\in PD^{w}_{k}(h^{*}\D,h^{*}\E)$ is a local PDO.

We say $P\in \pd_{k}^{2k}(S\times M\times \mathbb{R})$ is \emph{strongly elliptic} if there exists a constant $\lambda>0$ such that $B_{(k,k)}(t,x)(v,v)\geq \lambda\Vert v\Vert^{2}$ for all $t\in S,x\in M$ and $v\in T_{x}M$, i.e. $B_{(k,k)}(t,x)$ is a coercive bilinear form (this notion of coercivity for bilinear forms is different to the definition of locally coercive functions in Section \ref{eliassonsmethod}). It follows that in the local expression $A_{(k,k)}$ is also coercive, since $A_{(k,k)}(t,\xi)(\eta,\eta)=B_{(k,k)}(t,\exp_{h}\xi)(D_{2}\exp(\xi)\eta,D_{2}\exp(\xi)\eta)$.

\subsection{Finsler structures on $\H^{j,k}$}

On each fibre   $H^{j}(x^{*}TM),x\in C^{\infty}(S,M)$ we have an inner product:
\begin{equation}\label{ip}
\langle \xi,\eta \rangle_{j}:=\sum_{i=0}^{j}\int_{S}g(x)(\nabla_{t}^{i}\xi,\nabla_{t}^{i}\eta)
\end{equation}
where $g$ is the Riemannian metric on $M$,
with associated norm \eqref{norm}. These inner products can be extended to a Riemannian metric for $\H^{j,k}$, which we also denote by ${\langle\, ,\, \rangle}_{j}$. Again we construct the extension locally: let $\Phi_{h}$ be a local trivialisation over $(\phi_{h},U_{h})$ and suppose $x\in C^{\infty}(S,M)\cap U_{h}$. In order to write the local formula for the $H^{j}(x^{*}TM)$ inner product we will define a map $G:\mathcal{D}\to L(TM,TM)$ by 
\[ g(p)(v,G(u)w)=g(\exp u)(D_{2}\exp(u) v,D_{2}\exp(u) w)\]
where $u\in \mathcal{D}$ projects to $p\in M$ and $v,w\in T_{p}M$. Then $G(u)$ is positive and self-adjoint, and $G$ is a smooth fibre-preserving map. By Lemma 4.1 in \cite{Eluiasson:1967rr}, $G$ induces a smooth map $H^{k}(h^{*}\mathcal D)\to L(H^{k}(h^{*}TM),H^{k}(h^{*}TM)), \xi\mapsto G(\xi)$, which we also denote by $G$. Suppose now $x=\exp_{h}(\xi)$, and $\eta,\nu \in H^{j}(h^{*}TM)$ are the local representatives of $V,W\in H^{j}(x^{*}TM)$. Then 
\begin{align*}
{\langle V,W\rangle}_{j}&=\sum_{i=0}^{j}\int_{S}g(\exp_{h} (\xi))(\nabla_{t}^{i}D_{2}\exp (\xi)\eta,\nabla_{t}^{i}D_{2}\exp(\xi)\nu) \\
&=\sum_{i}\int_{S}g(h)(G(\xi){(\nabla_{t}^{i})}_{h}(\xi)\eta,{(\nabla_{t}^{i})}_{h}(\xi)\nu)\\
&=\sum_{i}{\langle G(\xi){(\nabla_{t}^{i})}_{h}(\xi)\eta,{(\nabla_{t}^{i})}_{h}(\xi)\nu \rangle}_{0}
\end{align*}
Both of $G$ and $(\nabla_{t}^{i})_{h}$ have extensions to smooth maps on $H^{k}(S,M)$, so the inner product \eqref{ip} can be extended to all $x\in H^{k}(S,M)$ producing a smooth Riemannian metric for $\H^{k,j}$. In particular the case $j=k$ gives a Riemannian metric for $H^{k}(S,M).$ It is proved in 
\cite{Eluiasson:1970cr} that the associated norm generates a Finsler structure, i.e. it satisfies  $\eqref{uniform}$, and that each connected component is a complete  metric space.

\begin{lemma}
The Finsler structure on $H^{k}(S,M)$ defined above is locally bounded with respect to $C^{k-1}(S,M).$
\end{lemma}
\begin{proof}
Let $\Phi_{h}^{-1}=(\exp_{h},D_{2}\exp)$ be a local trivialisation for $TH^{k}(S,M)=\H^{k,k}$, $\xi \in H^{k}(h^{*}\mathcal D)$, and $\eta\in H^{k}(h^{*}TM)$. Then from \eqref{localrep}
\begin{align*}
\Vert \Phi_{h}^{-1}(\xi,\eta)\Vert_{k}^{2}& =\sum_{i=0}^{k}\int_{S}g(h)( G(\xi){(\nabla_{t}^{i})}_{h}(\xi)\eta ,{(\nabla_{t}^{i})}_{h}(\xi)\eta)\\
& \leq \sum_{i}\int_{S}|G(\xi)|\Vert {(\nabla_{t}^{i})}_{h}(\xi)\eta \Vert^{2}\\
&\leq \const \sum_{i}\int_{S}|G(\xi)| (\Vert \nabla_{t}^{i}\eta\Vert^{2}+\sum_{j=0}^{i-1}|P_{ij}(\xi)|\Vert\nabla_{t}^{j}\eta\Vert^{2})
\end{align*}
 Now $\Vert \xi(t) \Vert\leq |\xi|_{0}\leq \const\Vert\xi\Vert_{k}$ by the Sobolev imbedding theorem, so if $\Vert \xi \Vert_{k}<L$ then $\xi(t)$ is contained in a compact subset of $\mathcal D$, and since $G, P_{ij}$ are continous we have that $|G(\xi)|, |P_{ij}(\xi)|$ are both bounded. Thus 
\[\Vert \Phi^{-1}(\xi,\eta)\Vert_{k}^{2}\leq \const {\Vert \eta \Vert}_{k}^{2} \]
for all $\Vert \xi\Vert_{k}\leq L$ and $\eta\in H^{k}(h^{*}TM)$, i.e.  $H^{k}(S,M)$ is locally bounded.
\end{proof}

\section{Some general results}
\label{mainresults}

The following theorem should be contrasted with Theorem 7 in \cite{Eliasson:1974sf}, where it is shown that given a strongly elliptic $P\in PD_{k}^{2k}(S\times M\times \mathbb R)$, the function $J(x):=\int_{S}P(x)$ is smooth, locally bounding and locally coercive with respect to $C^{0}$ on the submanifold $H^{k}(S,M)_{f}, f\in H^{k}(S,M)$ consisting of all $x\in H^{k}(S,M)$ with the same $H^{k}$ boundary conditions as $f$. Here we prove instead that $J$ is locally bounding and locally coercive on its natural domain $H^{k}(S,M)$. This is in order to accomodate a larger variety of variational problems, such as that considered in section \ref{cubicsintension} where the boundary conditions are of lower order than $P$. The cost of this generality is that we must weaken the statement to locally bounding and locally coercive with respect to $C^{k}$ instead of $C^{0}$.

\begin{theorem}\label{lblc}
Let $P$ be a strongly elliptic PDO on $S\times M\times \mathbb{R}$ of order $k$. Then the smooth function $J:H^{k}(S,M)\to \mathbb{R}$ defined by $J(x):=\int_{S}P(x)$ is locally bounding and locally coercive with respect to $C^{k-1}(S,M)$.
\end{theorem}

\begin{proof}
Let $\phi_{h}$ be a natural chart centred at $h$ and $P_{h}(\xi)$ the local expression for $P$, so 
\[ J_{h}(\xi)=\int_{S}P_{h}(\xi)=\sum_{\alpha} \int_{S} A_{\alpha}(\xi)(\nabla_{t}^{\alpha_{1}}\xi,\ldots,\nabla_{t}^{\alpha_{r}}\xi)
\]
For $\alpha\neq (k,k)$ let $X_{\alpha}=\int_{S}A_{\alpha}(\xi)(\nabla_{t}^{\alpha_{1}}\xi,\ldots,\nabla_{t}^{\alpha_{r}}\xi) $, then since ${|\xi|}_{0}$ is bounded so is $|A_{\alpha}(\xi)|$, and using the H\"older inequality $|X_{\alpha}|\leq \const {\Vert \nabla_{t}^{\alpha_{1}}\xi\Vert}_{0,p_{1}}\ldots{\Vert\nabla_{t}^{\alpha_{r}}\xi\Vert}_{0,p_{r}}$ where $\sum_{i}1/p_{i}=1$. Let us assume that $\alpha_{1}$ is the largest index and $p_{1}=2$. From the linear inclusion $C^{0}(h^{*}TM)\subset L^{p}_{0}(h^{*}TM)$ we have ${\Vert \nabla_{t}^{\alpha_{i}}\xi\Vert}_{0,p}\leq \const {|\nabla_{t}^{\alpha_{i}}\xi|}_{0}$ and therefore $|X_{\alpha}|\leq \const {\Vert \xi\Vert}_{k}{|\xi|}_{k-1}^{r}$. Now using the estimate $2ab\leq \varepsilon a^{2}+\tfrac{1}{\varepsilon}b^{2}$ for all $\varepsilon>0$, we obtain
\[ |X_{\alpha}|\leq \const\left (\varepsilon{\Vert \xi\Vert}_{k}^{2}+\tfrac{1}{\varepsilon}{|\xi|}_{k-1}^{2r}\right) \]
For $\alpha=(k,k)$, since $P$ is strongly elliptic there exists $\lambda'>0$ such that 
\[\int_{S}A_{k,k}(\xi)(\nabla_{t}^{k}\xi,\nabla_{t}^{k}\xi)\geq \lambda' {\Vert \nabla_{t}^{k}\xi\Vert}_{0}^{2} \]
Hence  $J_{h}(\xi)\geq \lambda' {\Vert \nabla_{t}^{k}\xi\Vert}_{0}^{2}-\const\left (\varepsilon{\Vert \xi\Vert}_{k}^{2}+\tfrac{1}{\varepsilon}{|\xi|}_{k-1}^{2r}\right)$, and choosing $\varepsilon$ sufficiently small there exists $\lambda>0$ such that 
\[ J_{h}(\xi)+\lambda'{\Vert\xi\Vert}_{k-1}^{2}\geq \lambda{\Vert\xi\Vert}_{k}^{2}-\const{|\xi|}_{k-1}^{2r} \]
It follows that if $J_{h}(\xi)$ and ${|\xi|}_{k-1}$ are bounded then so is ${\Vert \xi\Vert}_{k}$, i.e. $J$ is locally bounding with respect to $C^{k-1}$.

We now prove that $J$ is also locally coercive. For this we require the derivatives of $J_{h}$:
\begin{align*}
DJ_{h}(\xi)\eta & =\sum_{\alpha}\int_{S}D_{2}A_{\alpha}(\xi)(\eta,\nabla_{t}^{\alpha_{1}}\xi,\ldots,\nabla_{t}^{\alpha_{r}}\xi)\\
& \quad +\sum_{\alpha}\sum_{i=1}^{r}\int_{S}A_{\alpha}(\xi)(\nabla_{t}^{\alpha_{1}}\xi,\ldots,\nabla_{t}^{\alpha_{i}}\eta,\ldots,\nabla_{t}^{\alpha_{r}}\xi) \\
D^{2}J_{h}(\xi)(\eta,\eta) &=\int_{S}A_{k,k}(\xi)(\nabla_{t}^{k}\eta,\nabla_{t}^{k}\eta)+\sum_{i+j<2k}\int_{S}C_{ij}(\xi)(\nabla_{t}^{i}\eta,\nabla_{t}^{j}\eta)
\end{align*}
where the $C_{ij}$ are sums of derivatives of $A_{\alpha}, \alpha\neq (k,k)$. It will not be necessary to write down their precise expressions here; we merely note that that they are continuous functions of $\xi$ and its derivatives up to order $k$. Hence if ${\Vert\xi\Vert}_{k}$ is bounded then using the Cauchy-Schwarz inequality and strong ellipticity there exists $\lambda'>0$ such that
\begin{align*}
D^{2}J_{h}(\xi)(\eta,\eta)&\geq\lambda'{\Vert\nabla_{t}^{k}\eta\Vert}_{0}^{2}-\sum_{i+j<2k}\const{\Vert\nabla_{t}^{i}\eta\Vert}_{0}{\Vert\nabla_{t}^{j}\eta\Vert}_{0}\\
&\geq \lambda'{\Vert\nabla_{t}^{k}\eta\Vert}_{0}^{2}-\const{\Vert\eta\Vert}_{k}{\Vert\eta\Vert}_{k-1}\\
&\geq  \lambda'{\Vert\nabla_{t}^{k}\eta\Vert}_{0}^{2} -\const(\varepsilon{\Vert\eta\Vert}_{k}^{2}+\tfrac{1}{\varepsilon}{\Vert\eta\Vert}_{k-1}^{2})
\end{align*}
for all $\varepsilon>0$. Choosing $\varepsilon$ sufficiently small gives 
\[D^{2}J_{h}(\xi)(\eta,\eta)\geq\lambda{\Vert\eta\Vert}_{k}^{2}-c{\Vert\eta\Vert}_{k-1}^{2} \]
where $\lambda>0$ and $c$ are constant. The result now follows from the the imbedding $C^{k-1}\subset H^{k-1}$
\end{proof}

We now give some conditions under which the locally coercive and locally bounding properties persist upon restriction to a submanifold.
 
\begin{lemma}\label{restriction}
Let $X$ be a weak submanifold of $X_{0}$ and suppose $Y$ is a submanifold of $X$ such that for any $y\in Y$ there is a weak chart $(\phi,U)$ for $X$ containing $y$ which restricts to a weak chart for $Y$. Then $Y$ is a also weak submanifold of $X_{0}$ and if $f:X\to \mathbb R$ is locally bounding with respect to $X_{0}$ then so is $\tilde f:=f|Y$. If furthermore there exists a weak chart $(\phi,U)$ at $y$ which satisfies the submanifold property $\phi(U\cap Y)=\phi(U)\cap E\times \{0\}$, where $E$ is the model for $Y$, then if $f$ is locally coercive with respect to $X_{0}$  so is $\tilde f$.
\end{lemma}
\begin{proof}
Since $Y$ is a submanifold of $X$, $E$ splits in the model for $X$ and has the relative norm. Thus if $f$ is locally bounding so is $\tilde f$. The second statement follows from $D\tilde f_{\phi}=Df_{\phi}|_{E}$.
\end{proof}

The following results will be useful for but are not specific to the applications in Sections \ref{ce} and \ref{cubicsintension}.

\begin{lemma}\label{equicontinuous}
Let $U$ be a subset of $H^{k}(S,M)$ such that for all $x\in U$, ${\Vert \dot x\Vert}_{0} \leq K$ for some constant $K$. Then $U$ is an equicontinuous family of curves of bounded length.
\end{lemma}
\begin{proof}
For any $t_{1},t_{2}\in I, x\in U$ the H\"older inequality gives
\[ d(x(t_{1}),x(t_{2}))\leq \int_{t_{1}}^{t_{2}}{\Vert \dot x\Vert}\intd t\leq |t_{1}-t_{2}|^{\tfrac{1}{2}}K \]
Thus $U$ is equicontinuous and $\length(x)=\int_{0}^{1}\Vert\dot x\Vert \intd t \leq K$.
\end{proof}

\begin{lemma}
The energy function $E^{k}:H^{k}(S,M) \to \mathbb{R}, E^{k}(x):=\sob{\dot x}{k-1}^{2}$ is locally bounding with respect to $C^{0}(S,M)$.
\end{lemma}
\begin{proof}
By induction: $E^{k}$ is strongly elliptic and therefore locally bounding with respect to $C^{k-1}$ by Theorem \ref{lblc}, so $E^{1}$ is locally bounding with respect to $C^{0}$.
Suppose $E^{j}$ is locally bounding with respect to $C^{0}$, i.e. for any $x\in H^{j}(S,M)$ there exists a weak chart $(U_{h},\phi_{h})$ at $x$ such that for all $\xi\in U_{h}$, if ${|\xi|}_{0}$ and $E^{j}_{h}(\xi)$ are bounded then so is ${\Vert\xi\Vert}_{j}$. Then since $E^{j+1}\geq E^{j}$ it is sufficient to prove that if $E^{j+1}_{h}$ is bounded then so is $\sob{\nabla_{t}^{j+1}\xi}{0}$.

Now if $E^{j+1}(\phi_{h}^{-1}\xi)=\sob{\partial (\phi_{h}^{-1}\xi)}{j}^{2}\leq L$ for some constant $L$ then 
\[\sob{\nabla_{t}^{j}\partial (\phi_{h}^{-1}\xi)}{0}=\sob{\Phi_{h}^{-1}(\xi,(\nabla_{t}^{j}\partial)_{h}\xi)}{0}\leq L
\]
 and since $\sob{\,}{0}$ satisfies \eqref{uniform} it follows that $\sob{(\nabla_{t}^{j}\partial)_{h}\xi}{0}=\sob{\Phi_{h}(0,(\nabla_{t}^{j}\partial)_{h}\xi)}{0}$ is also bounded. From \eqref{lf} we have 
 \[\sob{{(\nabla_{t}^{j}\partial)}_{h}\xi}{0}=\sob{\nabla_{t}^{j+1}\xi+Q_{j+1}(\xi)}{0}\geq \sob{\nabla_{t}^{j+1}\xi}{0}-\sob{Q_{j+1}(\xi)}{0} \]
 where $Q_{j+1}\in \pd^{j+1}_{j}$, so now it suffices to prove $\sob{Q_{j+1}(\xi)}{0}$ bounded. Let $Q_{j+1}(\xi)=\sum_{\alpha}q_{\alpha}(\xi)(\nabla_{t}^{\alpha_{1}}\xi,\ldots,\nabla_{t}^{\alpha_{r}}\xi)$, where $\alpha_{i}\leq j$ and $\sum_{i}\alpha_{i}\leq j+1$, and let $Y_{\alpha}=\int_{S}\Vert q_{\alpha}(\xi)(\nabla_{t}^{\alpha_{1}}\xi,\ldots,\nabla_{t}^{\alpha_{r}}\xi)\Vert^{2}$ so that $\sob{Q_{j+1}(\xi)}{0}^{2}\leq \sum_{\alpha}Y_{\alpha}$. For each $Y_{\alpha}$ we have, by the H\"older inequality,
 \[ Y_{\alpha}\leq \const \sob{\nabla_{t}^{\alpha_{1}}\xi}{0,p_{1}}\ldots\sob{\nabla_{t}^{\alpha_{r}}\xi}{0,p_{r}} \]
 where $\sum_{i}\tfrac{1}{p_{i}}=1$. Let us assume that for all $i$ we have $\alpha_{1}\geq \alpha_{i}$, and choose $p_{1}=2$, so $\sob{\nabla_{t}^{\alpha_{1}}\xi}{0,p_{1}}\leq\sob{\xi}{j}$. Then for $i\neq 1$ we have $\alpha_{i}<j$, and 
 \[ \sob{\nabla_{t}^{\alpha_{i}}\xi}{0,p_{i}}\leq \const \co{\nabla_{t}^{\alpha_{i}}\xi}{0} \leq \const \co{\xi}{j-1} \]
via the imbedding $C^{0}\subset L^{p}_{0}$. Then we have $Y_{\alpha}\leq \const \sob{\xi}{j}\co{\xi}{j-1}^{r-1}\leq \const$, since $\sob{\xi}{j}$ is bounded (by assumption), and $\sob{Q_{j+1}(\xi)}{0}$ is bounded.
\end{proof}

\begin{corollary}\label{relcompact}
Suppose $U\subset H^{k}(S,M)$ is relatively compact in $C^{0}(S,M)$ and ${\Vert \dot x\Vert}_{k-1}\leq\const$ for all $x\in U$, then $U$ is relatively compact in $C^{k-1}(S,M)$.
\end{corollary}
\begin{proof}
Consider a sequence $(x_{i})\subset U$ which converges in $C^{0}(S,M)$ to $x\in \bar{U}$ and choose $h\in C^{\infty}(S,M)$ $C^{0}$-close to $x$ such that for all $i$ sufficiently large, $x_{i}$ is contained in the domain $U_{h}$ of the natural chart centred at $h$. Define the local sequence $(\xi_{i})$ by $x_{i}:=\exp_{h} \xi_{i}$. Then $(\xi_{i})$ is bounded in $C^{0}(h^{*}TM)$, and by the previous lemma ${\Vert \dot x \Vert}_{k-1}$ is locally bounding with respect to $C^{0}$. Thus we have ${\Vert \xi_{i}\Vert}_{k}$ bounded, and since the Sobolev imbedding $H^{k}(h^{*}TM)\hookrightarrow C^{k-1}(h^{*}TM)$ is compact, $(\xi_{i})$ converges in $C^{k-1}(h^{*}TM)$.
\end{proof}

\section{Higher order conditional extremals}\label{ce}

As our first application we will extend the results in \cite{Schrader:2012fk} proving the existence of conditional extremals joining two given submanifolds of $M$ provided one of the submanifolds is compact and the \emph{prior} vector field $A$ is bounded with respect to the Riemannian metric. Study of this function was initiated in \cite{Noakes:2012sf}, motivated by the problem of interpolating Riemannian manifold data obtained from an integral curve of an unknown vector field which is thought to be close to $A$. From the point of view of  mechanics, it is  more interesting to study the function 
\[ J_{2}(x):=\tfrac{1}{2}{\Vert \nabla_{t}\dot x-A(x)\Vert}^{2}_{0} \]
since the resulting minima will be curves which are $L^{2}$-close to satisfying Newton's equation $\nabla_{t}\dot x = A$. Here we will extend even further to higher order derivatives of $x$ and non-autonomous vector fields. Consider the restriction of 
\[ J_{k}(x):=\tfrac{1}{2}{\Vert \nabla_{t}^{k-1}\dot x-A(t,x)\Vert}^{2}_{0} \]
to the submanifold of $H^{k}(I,M)$ defined as follows.
Let $F$ be a map from $H^{k}(I,M)$ to the product of Whitney sums $\bigoplus_{k-1} TM\times \bigoplus_{k-1}TM$ defined by 
\[ F(x):=(x(0),\dot x(0),\ldots,\nabla_{t}^{k-2}\dot x(0),x(1),\ldots, \nabla_{t}^{k-2}\dot x(1)) \]
It can be shown that $F$ is a submersion, and therefore given any $v,w\in\bigoplus_{k-1} TM$ we have a submanifold $H^{k}(I,M)_{v,w}:=F^{-1}(v,w)$ consisting of curves which satisfy the boundary conditions $(x(0),\dot x(0),\ldots,\nabla_{t}^{k-2}\dot x(0))=v$ and $(x(1),\ldots, \nabla_{t}^{k-2}\dot x(1))=w$.

\begin{proposition}\label{conditional}
If ${\Vert  A(t,x)\Vert}$ is bounded then $J_{k}$ is weakly proper on $\hk_{v,w}$ with respect to $C^{k-1}(I,M)$.
\end{proposition}
\begin{proof}
Let $U$ be subset of $H^{k}(I,M)_{v,w}$ on which $J_{k}$ is bounded, i.e. there exist constants $\alpha, C$ such that for all $x\in U$ we have  $J_{k}(x)=\tfrac{1}{2}{\Vert \nabla_{t}^{k-1}\dot x-A(t,x)\Vert}_{0}^{2}\leq \alpha^{2}$, ${\Vert  A(t,x)\Vert}\leq C$, and therefore ${\Vert \nabla_{t}^{k-1}\dot x\Vert}_{0}\leq \alpha+C$. For any $i<k$ we have
\begin{align*}
{|\nabla_{t}^{i}\dot x|}_{0}^{2}={\Vert \nabla_{t}^{i}\dot x(s)\Vert}^{2}&={\Vert \nabla_{t}^{i}\dot x(0)\Vert}^{2}+\int_{0}^{s}\tfrac{d}{dt}\langle \nabla_{t}^{i}\dot x,\nabla_{t}^{i}\dot x\rangle\intd t \\
&\leq {\Vert v^{i}\Vert}^{2}+2\int_{0}^{1}|{\langle \nabla_{t}^{i+1}\dot x,\nabla_{t}^{i}\dot x\rangle}|\intd t \\
&\leq {\Vert v^{i}\Vert}^{2}+2{\Vert \nabla_{t}^{i+1}\dot x\Vert}_{0}{\Vert \nabla_{t}^{i}\dot x\Vert}_{0}
\end{align*}
by the Cauchy-Schwarz and H\"older inequalities. Dividing the above inequality by ${|\nabla_{t}^{i}\dot x|}_{0}$ and using $\Vert v^{i}\Vert,{\Vert \nabla_{t}^{i}\dot x\Vert}_{0}\leq {|\nabla_{t}^{i}\dot x|}_{0}$ we obtain
$ {\Vert\nabla_{t}^{i}\dot x\Vert}_{0}\leq \Vert v^{i}\Vert +2{\Vert \nabla_{t}^{i+1}\dot x\Vert}_{0}$ and then 
\[{\Vert\dot x\Vert}_{k-1}\leq \const \sum_{i=0}^{k-2}{\Vert v^{i}\Vert}+\const{\Vert \nabla_{t}^{k-1}\dot x\Vert}_{0} \]
Hence there exists $K$ such that ${\Vert \dot x\Vert}_{k-1}\leq K$, and in particular ${\Vert \dot x\Vert}_{0}\leq K$. By Lemma \ref{equicontinuous} $U$ is equicontinuous and each $x\in U$ has length at most $K$, and then since each $x\in U$ has the same initial point, $U(I)\subset M$ is bounded. Thus by the Arzel\'a-Ascoli theorem $U$ is relatively compact in $C^{0}(I,M)$, and by Corollary \ref{relcompact} $U$ is relatively compact in $C^{k-1}(I,M)$.
\end{proof}

\begin{theorem} \label{conditionalpsc}
$J_{k}$ satisfies the Palais-Smale condition on $\hk_{v,w}$
\end{theorem}
\begin{proof}
For any $x\in \hk_{v,w}$ we can choose $h\in C^{\infty}(I,M)\cap \hk_{v,w}$ such that $x$ is contained in the natural chart $(U_{h},\phi_{h})$ centred at $h$. This chart is a weak chart for $H^{k}(I,M)$ with $\phi_{h}(U_{h}\cap \hk_{v,w})=\phi_{h}(U_{h})\cap H^{k}(h^{*}TM)_{0}$, where $H^{k}(h^{*}TM)_{0}:=\{\xi\in H^{k}(h^{*}TM):\xi(0),\ldots,\nabla^{k-1}_{t}\xi(0)=\xi(1),\ldots,\nabla_{t}^{k-1}\xi(1)=0\}$ is the model space for $\hk_{v,w}$. Therefore by Lemma \ref{restriction}, Theorem \ref{lblc}, and Proposition \ref{conditional}, $J_{k}$ is locally bounding, locally coercive and weakly proper on $\hk_{v,w}$ with respect to $C^{k-1}$.
\end{proof}

\begin{corollary}
$J_{k}$ attains its infimum on $\hk_{v,w}$, and in any connected component there is a critical point which minimises $J_{k}$ with respect to the component.
\end{corollary}
\begin{proof}
This is a standard consequence of the Palais-Smale condition \cite{Palais:1988fv}.
\end{proof}

Note that if we set $A=0$ we reproduce the existence results for Riemannian cubics and critical points of higher order energy functions proved in \cite{Giambo:2002wd} and \cite{Giambo:2004nx}. 

The \emph{category} $\cat(X)$ of a topological space $X$ is a homotopy type invariant defined as the minimal number of closed contractible subsets of $X$ which cover $X$. Since the restriction of $J_{k}$ to $\hk_{v,w}$ satisfies the Palais-Smale condition it has at least $\cat (\hk_{v,w})$ critical points by the Lusternik-Schnirelmann multiplicity theorem (see eg. \cite{Palais:1988fv}). But it is shown in \cite{Giambo:2004nx} that $\hk_{v,w}$ has the same homotopy type as the based loop space $\Omega M$. Thus we have the following.

\begin{corollary}
There are at least $\cat (\Omega M)$ critical points of the restriction of $J_{k}$ to $\hk_{v,w}$. Moreover if $M$ is not contractible then there are infinitely many critical points.
\end{corollary}
\begin{proof}
If $M$ is not contractible then $\cat (\Omega M)$ is infinite \cite{Fadell:1991sf}.
\end{proof}

\section{Riemannian cubics in tension}\label{cubicsintension}
We now consider the problem of optimal $C^{1}$ interpolation of $n$ given points on a Riemannian manifold.  If we were to specify velocities at each of the given points then we can use piecewise Riemannian cubics, existence follows by setting $A=0$ in the previous section (or from \cite{Giambo:2002wd}). But if we do not then the Palais-Smale condition fails: suppose we have two points on a sphere and consider the sequence $(\gamma_{i})$ of geodesics which pass through the points and wrap around the sphere $i$ times. Each geodesic is a minimum of $J(x)=\tfrac{1}{2}\sob{\nabla_{t}\dot x}{0}^{2}$, but there is no convergent subsequence. The problem is that bounding $J$ does not bound the length of curves unless an initial velocity is specified, so without such a specification we can not prove $J$ is weakly proper. If we consider Riemannian cubics in tension instead then the cost function
\begin{equation}\label{rct}
 \mathcal J(x):=\tfrac{1}{2}({\Vert \nabla_{t}\dot x\Vert}_{0}^{2}+\tau^{2} {\Vert\dot x\Vert}_{0}^{2})
\end{equation}
dominates the length, and we will now show that the Palais-Smale condition holds. More precisely, given  points $p_{i}\in M, i=0,\ldots, n-1$ and corresponding times $t_{i}\in I$, we will prove that the Palais-Smale condition is satisfied when we restrict $\mathcal J$ to the submanifold  $\h2_{\mathbf p}:= \mathcal F^{-1}(\mathbf p)$, where $\mathbf p=(p_{0},\ldots, p_{n-1})$ and $\mathcal F:\h2\to M^{n}, x\mapsto (x(t_{0}),\ldots,x(t_{n-1}))$ is a submersion. Thus we will prove existence of solutions to what Hussein and Bloch refer to as the \emph{$\tau$-elastic variational problem without motion constraints} \cite{Hussein:2004ye, Hussein:2007rz}. 

\begin{lemma}
$\mathcal J$ is weakly proper on $H^{2}(I,M)_{\mathbf{p}}$ with respect to $C^{1}(I,M)$.
\end{lemma}
\begin{proof}
Suppose $\mathcal J(x)<K$ for all $x\in U\subset H^{2}(I,M)_{\mathbf{p}}$, then ${\Vert \nabla_{t}\dot x\Vert}_{0}^{2}\leq K$ and ${\Vert \dot x\Vert}_{0}^{2}\leq \tfrac{K}{\tau^{2}}$ and by Lemma \ref{equicontinuous}, $U$ is an equicontinous family of curves with length at most $K/\tau^{2}$. Thus since $x(0)=p_{0}$ for all $x\in U$, $U(I)$ is a bounded subset of $M$, and $U$ is relatively compact in $C^{0}(I,M)$ by the Arzel\'a Ascoli theorem. Then $U$ is also relatively compact in $C^{1}(I,M)$ by Corollary \ref{relcompact}, since ${\Vert \dot x\Vert}_{1}\leq \tau^{2}K$.
\end{proof}

\begin{theorem}
$\mathcal J$ satisfies the Palais-Smale condition on $\h2_{\mathbf p}$, therefore
attains its infimum on $\h2_{\mathbf p}$, and in any connected component there is a critical point which minimises $\mathcal J$ with respect to the component.
\end{theorem}
\begin{proof}
Just as in the proof of Theorem \ref{conditionalpsc} we can choose $h\in C^{\infty}(I,M)$ satisfying the interpolation conditions, and then the natural chart centred at $h$ restricts to a weak chart for each of $\h2$ and $\h2_{\mathbf p}$ which satisfies the submanifold property. Therefore, since the integrand is strongly elliptic, $\mathcal J$ is locally bounding and locally coercive on $\h2_{\mathbf p}$ with respect to $C^{1}$, and weakly proper with respect to $C^{1}$ by the previous Lemma. 
\end{proof}

Before applying Ljusternik-Schnirelman category theory we will relate the homotopy type of $\h2_{\mathbf p}$ to that of the based loop space $\Omega M$. For this we require the following lemma.

\begin{lemma}\label{fibration}
$\mathcal F:\h2_{\mathbf p}\to M^{n}, x\mapsto (x(t_{0}),\ldots,x(t_{n-1}))$ is a fibration.
\end{lemma}
The proof requires the following result of Earle and Eells \cite{Earle:1967fk}
\begin{proposition}(Earle and Eells \cite{Earle:1967fk})
Let $X,Y$ be Finsler manifolds modelled on Banach spaces, and suppose $X$ is complete. Let $f:X\to Y$ be a surjective map which foliates $X$. Furthermore suppose: for every $y\in Y$ there is a neighbourhood $V$ and a number $\eta>0$ such that for every $x\in f^{-1}(V)$ there is an $s_{x}\in L(T_{f(x)}Y,T_{x}X)$ such that $df_{x}\circ s_{x}=1_{T_{f(x)}Y}$, and $|s_{x}|\leq \eta$. Then $f$ is a locally $C^{0}$-trivial fibration.
\end{proposition}
\begin{proof}\emph{(Lemma \ref{fibration}).} For each $t_{i}$ we choose an open interval $B_{\epsilon_{i}}(t_{i})$ such that $\cap_{i}B_{\epsilon_{i}}(t_{i})=\emptyset$, and such that $x(\overline{B_{\epsilon_{i}}}(t_{i}))$ is contained in the domain of a chart $\phi_{i}$. Let $b_{i}:S\to \mathbb R$ be a smooth map such that $b_{i}(t_{i})=1$ and $\supp b_{i}\subset B_{\epsilon_{i}}(t_{i})$. Given $\mathbf v\in T_{\mathcal F(x)}M^{n}$ we define
\[(s_{x}\mathbf v)(t)=\left\{
        \begin{array}{ll}
            b_{i}(t_{i})d\phi_{i}(x(t))^{-1}d\phi_{i}(x(t_{i}))v_{i} & \quad t\in B_{\epsilon_{i}}(t_{i})  \\
            0 & \quad $elsewhere$
        \end{array}
    \right.
\]
Then $dE\circ s_{x}\mathbf v=\mathbf v$ and $|s_{x}\mathbf|$ is bounded.

\end{proof}

\begin{theorem}
The inclusion $\hk_{\mathbf p}\to C^{0}(I,M)_{\mathbf p}$ is a homotopy equivalence.
\end{theorem}
\begin{proof}
It is well known that $C^{0}(I,M)\to M^{n}$ is a fibration when $n=2$ (eg. \cite{Spanier:1966pb} p. 98), and a similar proof shows it is a fibration for any $n$. The inclusion $H^{k}(I,M)\to C^{0}(I,M)$ is a homotopy equivalence by a theorem of Palais (\cite{Palais:1968fy} Theorem 13.14). If we apply the five lemma to the homotopy sequence for the fibrations
\[
\begin{CD}
H^{k}(I,M) @>>> C^{0}(I,M)\\
@VV V @VV V\\
M^{n} @=  M^{n}
\end{CD}
\]
we see that inclusion of the fibres $H^{k}(I,M)_{\mathbf p}\to C^{0}(I,M)_{\mathbf p}$ induces isomorphisms of homotopy groups, and is therefore a homotopy equivalence.
\end{proof}

\begin{corollary}
$H^{k}(I,M)_{\mathbf p}$ has the same homotopy type as the $n$-fold cartesian product of the based loop space: $(\Omega M)^{n}$.
\end{corollary}
\begin{proof}
$C^{0}(I,M)_{\mathbf p}$ is homeomorphic to $C^{0}(I,M)_{p_{0},p_{1}}\times\ldots \times C^{0}(I,M)_{p_{n-2},p_{n-1}}$, and each $C^{0}(I,M)_{p,q}$ has the same homotopy type as $\Omega M$.
\end{proof}

\begin{corollary}
The restriction of $\mathcal J$ to $\h2_{\mathbf p}$ has at least $\cat((\Omega M)^{n})$ critical points. 
\end{corollary}

\paragraph{Acknowledgements}
I am very grateful to Prof. Lyle Noakes for many helpful conversations.

\bibliographystyle{elsarticle-harv}
\bibliography{/Users/phil/Dropbox/TeX/refs}

\end{document}